\RequirePackage{fix-cm}
\documentclass[smallextended]{svjour3}       
\smartqed  

\newtheorem{rem}{Remark}
\newtheorem{exa}{Example}
\usepackage{graphicx}
\usepackage{amssymb}

\usepackage{enumerate}

\begin{document}

\title{On the simultaneous Diophantine equations $ m \cdot (x_1^k+x_2^k+ \cdots + x_{t_1}^k)=n \cdot (y_1^k+y_2^k+ \cdots y_{t_2}^k)$; \ $k=1,3$
}

\titlerunning{$ m \cdot (x_1^k+x_2^k+ \cdots + x_{t_1}^k)=n \cdot (y_1^k+y_2^k+ \cdots y_{t_2}^k)$; \ $k=1,3$}        

\author{Farzali Izadi         \and Mehdi Baghalaghdam
}


\institute{Farzali Izadi \at
Department of Mathematics\\ Faculty of Science\\ Urmia University\\ Urmia 165-57153, Iran\\
\email{f.izadi@urmia.ac.ir} 
\and
Mehdi Baghalaghdam \at
Department of Mathematics \\ Faculty of Science \\ Azarbaijan Shahid Madani University \\ Tabriz 53751-71379, Iran\\
\email{mehdi.baghalaghdam@yahoo.com}
}

\date{Received: date / Accepted: date}

\maketitle

\begin{abstract}
In this paper, we solve the simultaneous Diophantine equations $m \cdot( x_{1}^k+ x_{2}^k +\cdots+ x_{t_1}^k)=n \cdot (y_{1}^k+ y_{2}^k +\cdots+ y_{t_2}^k )$,\ $k=1,3$, where $ t_1, t_2\geq3$, and $m$, $n$ are fixed arbitrary and relatively prime positive integers. This is done by choosing two appropriate trivial parametric solutions and obtaining  infinitely many nontrivial parametric solutions.
  Also we work out some examples, in particular the Diophantine systems of $A^k+B^k+C^k=D^k+E^k$, $k=1,3$.
\keywords{ Simultaneous Diophantine equations \and  Cubic Diophantine equations \and Equal sums of cubes.}
 \subclass{Primary11D45 \and  Secondary11D72 \and 11D25.}
\end{abstract}

\section{Introduction}
\label{intro} The cubic Diophantine equations has been studied by some mathematicians. The Diophantine equation $x^3+y^3=z^3$, which is a special case of Fermat's last theorem, has  no solution in integers.\\ Gerardin gave partial solutions of the simultaneous Diophantine equation (SDE)
\begin{equation}\label{(1)}
x^k+y^k+z^k=u^k+v^k+w^k; k=1,3
\end{equation}

\noindent in 1915-16 (as quoted by Dickson, pp. 565, 713 of \cite{4})
 and additional partial solutions were given by Bremner \cite{1}. Subsequently, complete solutions were given
in terms of cubic polynomials in four variables by Bremner and Brudno \cite{2},
as well as by Labarthe \cite{6}.\\

\noindent  In \cite{3} Choudhry presented a complete four-parameter solution of ($1$)
in terms of quadratic polynomials in which each parameter occurs only in the first degree.\\

\noindent The authors in a different paper used two different methods to solve the cubic Diophantine equation $ x_{1}^3+ x_{2}^3 +\cdots+ x_{n}^3=k \cdot (y_{1}^3+ y_{2}^3 +\cdots+ y_{\frac{n}{k}}^3 )$, where $ n \geq3$, and $k \neq n$, is a divisor of $n$ ($\frac{n}{k}\geq2$), and obtained infinitely many nontrivial parametric solutions (see \cite{5}).\\

\noindent In this paper, we are interested in the study of the SDE:

\begin{equation}\label{(2)}
m \cdot( x_{1}^k+ x_{2}^k + \cdots +x_{t_1}^k)=n \cdot (y_{1}^k+ y_{2}^k + \cdots + y_{t_2}^k ), \\ k=1,3
\end{equation}

\noindent where $ t_1, t_2\geq3$, $n$, $m$  are  fixed arbitrary and relatively prime positive integers. To the best of our knowledge ($2$) has not already been considered by any other authors.
\\
\section{Main result}
\noindent  We prove the following main theorem:

\begin{theorem}
Let $ t_1, t_2\geq3$, $n$, $m$ be  fixed arbitrary and relatively prime positive integers. Then ($2$)  has infinitely many nontrivial  parametric solutions in non-zero integers $x_1$, $\cdots$, $x_{t_1}$, $y_1$, $\cdots$, $y_{t_2}$.
\end{theorem}

\begin{proof}:
\noindent Firstly, it is clear that if\\

\noindent  $X=(x_1, \cdots ,x_{t_1},y_1, \cdots ,y_{t_2})$, and
\noindent $ Y=(X_1, \cdots ,X_{t_1},Y_1, \cdots ,Y_{t_2}), $\\

\noindent  are two solutions of ($2$), then for any arbitrary rational numbers $t$, $X+tY$ is also a solution for $k=1$, i.e.,\\

\noindent $m \cdot[(x_1+tX_1)+(x_2+tX_2)+ \cdots +(x_{t_1}+tX_{t_1})]=\\
n \cdot [(y_1+tY_1)+(y_2+tY_2)+ \cdots +(y_{t_2}+tY_{t_2})]$.\\

\noindent We say that $X$ is a trivial parametric solution of ($2$), if it is nonzero and satisfies the system trivially.   
Let $X$, and $Y$  be two proper trivial parametric solutions of ($2$). (we will introduce them later.)

\noindent We suppose that $ X+tY$, is also a solution for the case of $k=3$, where $t$ is a parameter. We wish to find $t$.
\\

\noindent By plugging\\

\noindent $X+tY=(x_1+tX_1,x_2+tX_2, \cdots ,x_{t_1}+tX_{t_1},y_1+tY_1, y_2+tY_2, \cdots ,y_{t_2}+tY_{t_2})$,
\\

\noindent into ($2$), we get:\\

\noindent $m \cdot[(x_1+tX_1)^3+(x_2+tX_2)^3+ \cdots +(x_{t_1}+tX_{t_1})^3]=\\
n \cdot [(y_1+tY_1)^3+(y_2+tY_2)^3+ \cdots +(y_{t_2}+tY_{t_2})^3]$.
\\

\noindent Since $X$ and $Y$ are solutions of ($2$), after some simplifications, we obtain:
\\

\noindent $3t^2(mx_1X_1^2+mx_2X_2^2+ \cdots +mx_{t_1}X_{t_1}^2-ny_1Y_1^2-ny_2Y_2^2- \cdots -ny_{t_2}Y_{t_2}^2)+$\\

\noindent $3t(mx_1^2X_1+mx_2^2X_2+ \cdots +mx_{t_1}^2X_{t_1}-ny_1^2Y_1-ny_2^2Y_2- \cdots -ny_{t_2}^2Y_{t_2})=0.$\\

\noindent Therefore $t=0$ or

 $$t=\frac{mx_1^2X_1+ \cdots +mx_{t_1}^2X_{t_1}-ny_1^2Y_1- \cdots -ny_{t_2}^2Y_{t_2}}{-mx_1X_1^2- \cdots -mx_{t_1}X_{t_1}^2+ny_1Y_1^2+ \cdots +ny_{t_2}Y_{t_2}}:=\frac{A}{B}.$$
\\

\noindent By substituting $t$ in the above expressions, and clearing the denominator $B^3$, we get an integer solution of ($2$) as follows:
\\










\noindent $(x'_1,x'_2, \cdots ,x'_{t_1},y'_1,y'_2, \cdots ,y'_{t_2})=$\\

\noindent $(x_1B+AX_1,x_2B+AX_2, \cdots ,x_{t_1}B+AX_{t_1},y_1B+AY_1,y_2B+AY_2, \cdots ,y_{t_2}B+AY_{t_2})$.\\


\noindent If we pick up trivial parametric solutions $X$ and $Y$ properly, we will get a nontrivial parametric solution of ($2$).
\noindent We should mention that not every trivial parametric solutions of $X$ and $Y$ necessarily give rise to a nontrivial parametric solution of ($2$). So the trivial parametric solutions must be chosen properly.\\

\noindent Now we introduce the  proper trivial parametric solutions.\\
\noindent For the sake of simplicity, we only write down the trivial parametric solutions for the left hand side of ($2$). It is clear that the trivial parametric solutions for the right hand side of ($2$) can be similarly found by the given trivial parametric solutions of the left hand side of ($2$).

\noindent Let $p_i$, $q_i$, $s_i$, $r_i$$\in\mathbb{Z}$.\\

\noindent There are $4$  different possible cases for $t_1$:\\

\noindent $1$. $t_1=2\alpha+1$, $\alpha$ is even.\\

\noindent $X_{left}=(x_1,x_2, \cdots ,x_{t_1})=
(p_1,-p_1,p_2,-p_2, \cdots ,p_\alpha,-p_\alpha, 0)$, and
\\

\noindent $Y_{left}=(X_1,X_2, \cdots ,X_{t_1}) =$
\\
\noindent $(r_1,r_2,-r_1,-r_2,r_3,r_4,-r_3,-r_4, \cdots ,r_{\alpha-1},r_\alpha,-r_{\alpha-1},0,-r_\alpha).$
\\

\noindent $2$. $t_1=2\alpha+1$, $\alpha$ is odd.\\

\noindent $X_{left}=(p_1,-p_1,p_2,-p_2, \cdots ,p_\alpha,-p_\alpha,0)$, and\\

\noindent $Y_{left}=(r_1,r_2,-r_1,-r_2, \cdots ,- r_{\alpha-2},-r_{\alpha-1},r_\alpha,0,-r_\alpha)$.\\

\noindent $3$. $t_1=2\alpha$, $\alpha$ is even.\\

\noindent $X_{left}=(p_1,-p_1,p_2,-p_2, \cdots ,p_\alpha,-p_\alpha)$, and\\

\noindent $Y_{left}=(r_1,r_2,-r_1,-r_2, \cdots , r_{\alpha-1},r_\alpha,-r_{\alpha-1},-r_\alpha).$\\

\noindent $4$. $t_1=2\alpha$, $\alpha$ is odd.\\

\noindent $X_{left}=(p_1,-p_1,p_2,-p_2, \cdots ,p_\alpha,-p_\alpha)$, and\\

\noindent $Y_{left}=(r_1,r_2,-r_1,-r_2, \cdots , r_{\alpha-2},r_{\alpha-1},-r_{\alpha-2},r_\alpha,-r_{\alpha-1},-r_\alpha).$\\

\noindent Finally, it can be easily shown that for every $i \neq j$, we have:
\\

\noindent  $x_iB+AX_i$ $\neq\pm$ $(x_jB+AX_j)$, $y_iB+AY_i$ $\neq\pm$ $(y_jB+AY_j)$ and\\ $x_iB+AX_i$$\neq\pm$ $(y_jB+AY_j)$, \\

\noindent i.e., it is really a nontrivial parametric solution in terms of the parameters $p_i$, $q_i$, $r_i$ and $s_i$.\\
\noindent It is clear that by fixing  all of the parameters $p_i$, $q_i$, $r_i$ and $s_i$, but one parameter, we can obtain a nontrivial one parameter parametric solution of ($2$) and by changing properly the fixed values, finally get infinitely many nontrivial one parameter parametric solutions of ($2$). 

\noindent Now, the proof of the theorem is completed. $\spadesuit$\\
\end{proof}
\noindent In  the remaining  part, we worked out some examples.

\section{Application to examples}

\begin{exa} $m \cdot (x_1^k+x_2^k+x_3^k)=n \cdot(y_1^k+y_2^k+y_3^k);$ $k=1,3$.\\

\noindent Trivial parametric solutions (case $2$):\\

\noindent $X=(x_1,x_2,x_3,y_1,y_2,y_3)=$
\noindent $(p_1,-p_1,0,q_1,-q_1,0)$,\\

\noindent $Y=(X_1,X_2,X_3,Y_1,Y_2,Y_3)=$
\noindent $(r_1,0,-r_1,s_1,0,-s_1)$,\\

\noindent$A=mp_1^2r_1-nq_1^2s_1$,
\noindent$B=-mp_1r_1^2+nq_1s_1^2$,\\

\noindent  nontrivial parametric solution:\\
\noindent  $x_1=p_1B+r_1A,$\\
\noindent  $x_2=-p_1B,$\\
\noindent  $x_3=-r_1A,$\\
\noindent $y_1=q_1B+s_1A,$\\
\noindent $y_2=-q_1B,$\\
\noindent $y_3=-s_1A.$\\

\noindent Numerical example:\\
\noindent $p_1=4$, $q_1=1$, $r_1=2$,
 $s_1=3$,\\

\noindent Solution:\\

\noindent$m \cdot  [(30n)^k+(64m-36n)^k+(-64m+6n)^k]=\\n \cdot[(80m)^k+ (16m-9n)^k+(-96m+9n)^k
];$ $k=1,3$.\\

\end{exa}

\begin{exa} $m \cdot (x_1^k+x_2^k+x_3^k+x_4^k)=n \cdot(y_1^k+y_2^k+y_3^k +y_4^k);$ $k=1,3$.\\

\noindent Trivial parametric solutions (case $3$):\\

\noindent $X=(x_1,x_2,x_3,x_4,y_1,y_2,y_3,y_4)=$
\noindent $(p_1,-p_1,p_2,-p_2,q_1,-q_1,q_2,-q_2)$,\\

\noindent $Y=(X_1,X_2,X_3,X_4,Y_1,Y_2,Y_3,Y_4)=$
\noindent $(r_1,r_2,-r_1,-r_2,s_1,s_2,-s_1,-s_2)$,\\

\noindent Numerical example:\\
\noindent $p_1=2$, $p_2=5$, $q_1=1$, $q_2=3$, $r_1=6$, $r_2=7$,
 $s_1=4$, $s_2=9$,\\

\noindent Solution:\\

\noindent$m \cdot  [(-1456m+104n)^k+(-2093m+1248n)^k+(2093m-1924n)^k+(1456m+572n)^k]=n \cdot[(-1001m+156n)^k+ (-2548m+1196n)^k+(1365m-1196n)^k+(2184m-156n)^k
];$ $k=1,3$.\\

\end{exa}

\begin{exa}
 $x_1^k+x_2^k+ \cdots +x_5^k=n \cdot (y_1^k+y_2^k+ \cdots +y_5^k);\  k=1,3.$
\\

\noindent Trivial parametric solution (case $2$):
\\

\noindent  $X=(x_1,x_2,x_3,x_4,x_5,y_1,y_2,y_3,y_4,y_5)=$
$(p_1,-p_1,p_2,-p_2,0,q_1,-q_1,q_2,-q_2,0)$,
\\

\noindent $Y=(X_1,X_2,X_3,X_4,X_5,Y_1,Y_2,Y_3,Y_4,Y_5)=$
$(r_1,r_2,-r_1,0,-r_2,s_1,s_2,-s_1,0,-s_2)$,
\\

\noindent $A=p_1^2r_1+p_1^2r_2-p_2^2r_1-nq_1^2s_1-nq_1^2s_2+nq_2^2s_1$,
\\

\noindent $B=-p_1r_1^2+p_1r_2^2-p_2r_1^2+nq_1s_1^2-nq_1s_2^2+nq_2s_1^2,$\\

\noindent  nontrivial parametric solution:\\

\noindent $x_1=x_1B+X_1A=p_1B+r_1A$,\\

\noindent $x_2=x_2B+X_2A=-p_1B+r_2A$,\\

\noindent $x_3=x_3B+X_3A=p_2B-r_1A$,\\

\noindent $x_4=x_4B+X_4A=-p_2B$,\\

\noindent $x_5=x_5B+X_5A=-r_2A$,\\

\noindent  $y_1=y_1B+Y_1A=q_1B+s_1A$,\\

\noindent $y_2=y_2B+Y_2A=-q_1B+s_2A$,\\

\noindent  $y_3=y_3B+Y_3A=q_2B-s_1A$,\\

\noindent  $y_4=y_4B+Y_4A=-q_2B$,\\

\noindent $y_5=y_5B+Y_5A=-s_2A$.\\

\noindent Numerical example $1$:\\
\noindent $p_1=5$, $p_2=6$, $q_1=7$, $q_2=8$, $r_1=1$, $r_2=2$, $s_1=3$, $s_2=4$,\\

\noindent Solution:\\

\noindent$(84-36n)^k+(33-417n)^k+(15+289n)^k+(-54-138n)^k+(-78+302n)^k=n \cdot[  (180-292n)^k+(93-765n)^k+(-45+637n)^k+(-72-184n)^k+(-156+604n)^k];$ $k=1,3$.
\\

\noindent $n=0$ $\Longrightarrow$  $5^3+11^3+28^3=18^3+26^3$.\\
It is interesting to see that $5+11+28=18+26$, too.

\begin{rem} By fixing the values  $p_2$, $r_1$, $r_2$, $n=0$, and letting $p_1=:p$, as a parameter, we can obtain infinitely many nontrivial parametric solutions for the simultaneous Diophantine equations of the form $A^k+B^k+C^k=D^k+E^k$, $k=1,3$.
\end{rem}

\noindent Numerical example $2$:\\
\noindent $n=0$, $r_1=1$, $r_2=3$, $p_2=5$, $p_1=:p$,\\

\noindent Solution:\\

\noindent$(12p^2-5p-25)^k+(4p^2+5p-75)^k+(-4p^2+40p)^k=\\ (40p-25)^k+(12p^2-75)^k;  k=1,3$.

\end{exa}

\begin{acknowledgements}
We are very grateful to the referee for the careful reading of the paper and giving several useful comments which improved the quality of the paper.
\end{acknowledgements}



\end{document}